\documentclass{article}
\usepackage[journal=RLM,lang=british]{ems-journal-rlm} 
\usepackage[margin=0pt]{subfig}

 \theoremstyle{plain}
          \newtheorem{theorem}{Theorem}[section]
          \newtheorem{lemma}[theorem]{Lemma}
           \newtheorem{definition}[theorem]{Definition}
\theoremstyle{definition}
          
          \newtheorem{remark}[theorem]{Remark}

\numberwithin{equation}{section}

\begin{document}

\scientificchapter{Mathematical Analysis}

\title{Trace operators for Riemann--Liouville fractional  equations}
\titlemark{Trace operators for Riemann--Liouville fractional  equations}


\emsauthor{1}{
	\givenname{Paola}
	\surname{Loreti}
	}
	{P.~Loreti}
\emsauthor*{2}{
	\givenname{Daniela}
	\surname{Sforza}
	}
	{D.~Sforza}

\Emsaffil{1}{
	\department{Dipartimento di Scienze di Base e Applicate per l'Ingegneria}
	\organisation{Sapienza Universit\`a di Roma}
	\address{via Antonio Scarpa 16}
	\zip{00161}
	\city{Roma}
	\country{Italy}
	\affemail{paola.loreti@uniroma1.it}
	}
\Emsaffil{2}{
	\department{Dipartimento di Scienze di Base e Applicate per l'Ingegneria}
	\organisation{Sapienza Universit\`a di Roma}
	\address{via Antonio Scarpa 16}
	\zip{00161}
	\city{Roma}
	\country{Italy}
	\affemail{daniela.sforza@uniroma1.it}
	}


\classification[26A33]{35R11}


\begin{abstract}
We begin with a brief overview of the most commonly used fractional derivatives, namely the Caputo and Riemann-Liouville derivatives. We then focus on the study of the fractional time wave equation with the Riemann-Liouville derivative, addressing key questions such as well-posedness, regularity, and a trace result in appropriate interpolation spaces. Additionally, we explore the duality relationship with the Caputo fractional time derivative. The analysis is based on expanding the solution in terms of Mittag-Leffler functions.
\end{abstract}

\maketitle


\section{Introduction}\label{s:int}   

%
%
%
%
%

%

Although theoretical results concerning fractional differential equations are still being studied by mathematicians -- including well-posedness in Sobolev spaces, regularity properties, asymptotic behavior of solutions, and so on -- the modeling and consistency with real phenomena drive the applications of fractional calculus. This is particularly true in engineering and science, especially in fields like rheology, viscoelasticity, control theory, bioengineering, and many others, see e.g.  \cite{Hilfer,Pov}.

One of the fascinating  aspects of this topic is that the definition of fractional derivatives relies on  non-local operators such as
 Riemann-Liouville integrals of the type  
\begin{equation*}
I^{\alpha }u(t)=\frac{ 1}{\Gamma (\alpha )}\int _{0}^{t}(t-s)^{\alpha -1}u(s)\,ds,
\end{equation*}
where  $\Gamma$ is the Gamma function, see \cite{Liou1,Liou2,Ri}.
The integral is well defined provided that $\alpha>0$
and $u$ is a locally integrable function. 
The operator defined by the Riemann-Liouville integral is non-local in nature. This means that the value of the fractional integral at a point $t$ depends not only on the value of the function $u$ at $t$, but also on values of the function at all points $s\in(0,t)$. This non-local behavior is a direct consequence of the kernel $s^{\alpha-1}$, which gives the integral a memory-like property. This is in contrast to classical integer-order calculus, where the derivative or integral at a point depends only on local values. The structure of the kernel $s^{\alpha-1}$ plays a crucial role in defining the behavior of the fractional integral. For example, in the case $\alpha\in(0,1)$ the kernel causes the fractional integral to be singular at $s=t$.

It is noteworthy to mention the paper \cite{DaP}, where G. Da Prato et al. address singularities at the origin, which are just characterized by the singular kernel 
$t^{-\beta}$, where $0<\beta<1$.





The existing literature on fractional calculus primarily focuses on the Riemann-Liouville derivative $\sideset{}{^{\alpha}}{\mathop D}u$ and 
Caputo derivative $\sideset{^{C}}{^{\alpha}}{\mathop D}u$ \cite{Caputo,GM}, defined for $\alpha\in(1,2)$ as
\begin{equation*}
\sideset{}{^{\alpha}}{\mathop D}u
=\Big(I^{2-\alpha}u\Big)_{tt},
\qquad
\sideset{^{C}}{^{\alpha}}{\mathop D}u
=I^{2-\alpha}u_{tt}.
\end{equation*}
 One of the key advantages of the Caputo fractional derivative is the more straightforward formulation of initial conditions. In traditional fractional calculus, initial conditions often involve fractional derivatives, which can be complex to interpret. However, the Caputo derivative is defined in such a way that the initial conditions can be specified in a more conventional manner, similar to integer-order derivatives, making it easier to apply in practical problems.

The analysis of fractional PDEs mainly focuses on the fractional Caputo equation due to its more regularity and simplicity in the formulation of the initial conditions. Although a connection between the two problems can be established (see Remark \ref{re:RLversusC}), we opt for a direct analysis due to the complications arising from the substitution.


We initiated our study of trace regularity for integro-differential problems in the context of regular kernels, as discussed in \cite{LoretiSforza2, LoretiSforza3}. In these works, we establish a hidden regularity result for wave equations that include a convolution-type integral term. 
We introduce the concept of the trace of the normal derivative of the solution, thereby extending well-known results from linear wave equations without memory to integro-differential equations, even when nonlinear terms are present. Specifically, for the case without memory, we refer to \cite{LasTri}, where the authors found out a trace theory interpretation for solutions of
hyperbolic systems, and \cite{Lio2}, where the term "hidden regularity" was first introduced in the context of semilinear wave equations.

%

In \cite{LoretiSforza4} we prove a hidden regularity result for weak solutions of time fractional diffusion-wave equations with the Caputo fractional derivative of order $\alpha\in(1,2)$. Moreover, in \cite{LoretiSforza5} we introduce a notion of weak solution for
abstract fractional differential equations with Caputo derivative, also providing 
two examples of concrete equations: time-fractional
wave equations and time-fractional Petrovsky systems.


Important tools for representing the solutions of fractional differential equations, already in the scalar case, are the Mittag-Leffler functions.

In \cite{ML0,ML} the foremost mathematician G.M. Mittag-Leffler introduced  for $\alpha>0$
the entire function
\begin{equation}\label{eq:MLp}
E_{\alpha}(z)=\sum _{k=0}^{\infty }\frac {z^k}{\Gamma (\alpha k+1)},
\quad z\in\mathbb{C},
\end{equation}
called after him the Mittag-Leffler function. Subsequently, in the paper
\cite{ML1} he extended the definition \eqref{eq:MLp} to the case $\alpha\in\mathbb{C}$ with $\Re\alpha>0$.
The Mittag-Leffler function can be considered a direct generalisation of the exponential function, retaining some of its properties. 
In \cite{W} A. Wiman also noted that analogous asymptotic results hold for the two-parametric generalisation $E_{\alpha ,\beta }(z)$ 
of the Mittag-Leffler function
defined as
\begin{equation*}
E_{\alpha,\beta }(z)=\sum _{k=0}^{\infty }\frac {z^k}{\Gamma(\alpha k+\beta )},
\quad z\in\mathbb{C}.
\end{equation*}
In the 20th century, this function was virtually unknown to most scientists, as it was ignored in most books on special functions. 
 The recent growing interest in this function is mainly due to its close relationship with fractional calculus and especially with fractional problems arising from applications.
For further reading about the Mittag-Leffler functions, we refer to the book \cite{GKMR}.

In this paper we are interested in the fractional Riemann-Liouville problem:
\begin{equation}\label{eq:wI}
\begin{aligned}
\sideset{}{_{0+}^{\alpha}}{\mathop D}u
&=\Delta u
\qquad\qquad
\hbox{in}\ (0,T)\times\Omega,
\\
u&=0 
\qquad\qquad 
\hbox{on}\  (0,T)\times\partial\Omega,
\end{aligned}
\end{equation}
where $\Omega\subset\mathbb{R}^N$,  $N\ge1$, is a bounded open set with $C^2$ boundary $\partial\Omega$ and  the order $\alpha\in\big(\frac32,2\big)$.
Our main results are the following. First, we introduce a notion of weak solution and establish an existence result.
For $u_1\in L^2(\Omega)$ and $u_2\in H_0^1(\Omega)$  the function
\begin{equation*}
u(t,x)
=\sum_{n=1}^\infty\big[ \langle u_1,e_n\rangle t^{\alpha-1}E_{\alpha,\alpha}(-\lambda_nt^\alpha)
+\langle u_2,e_n\rangle t^{\alpha-2} E_{\alpha,\alpha-1}(-\lambda_nt^\alpha)\big]e_n(x)
\end{equation*}
is the unique weak solution of \eqref{eq:wI} written as series, belonging to $L^2(0,T;H^1_0(\Omega))$ and  satisfying
\begin{equation*}
\sideset{}{_{0+}^{\alpha-1}}{\mathop D}u(0,\cdot)=u_{1},\qquad \sideset{}{_{0+}^{\alpha-2}}{\mathop D}u(0,\cdot)=u_{2},
\end{equation*}
see Theorem \ref{th:exist}.
Next, we prove a regularity result when the initial data belong to interpolation spaces connected to the Laplace operator:
for $u_1, \nabla u_2\in D((-\Delta)^{\mu_\alpha})$, $\mu_\alpha\ge0$, we have:
\begin{enumerate}[(i)]
\item 
for $\theta\in\big(\mu_\alpha,\frac{2\alpha-3}{2\alpha}+\mu_\alpha\big)$
\begin{equation*}
\lVert\nabla u\rVert_{L^2(0,T;D((-\Delta)^{\theta}))}\le C\big(\lVert u_1\rVert_{D((-\Delta)^{\mu_\alpha})}+\lVert\nabla u_2\rVert_{D((-\Delta)^{\mu_\alpha})}\big),
\end{equation*}
\item 
for $\theta\in\big(\frac{3-\alpha}{2\alpha}-\mu_\alpha,\frac12-\mu_\alpha\big)$ 
\begin{equation*}
\lVert\sideset{}{_{0+}^{\alpha}}{\mathop D}u\rVert_{L^2(0,T;D(A^{-\theta}))}\le C\big(\lVert u_1\rVert_{D((-\Delta)^{\mu_\alpha})}
+\lVert\nabla u_2\rVert_{D((-\Delta)^{\mu_\alpha})}\big),
\end{equation*}
\end{enumerate}
for some constant $C>0$, see Theorem \ref{th:reg-l2}.

A key step in proving the trace regularity result (see Theorem \ref{th:hidalpha}) is to ensure that the above estimates hold for the same value of $\theta$.
To guarantee that the intersection of the intervals $\big(\mu_\alpha,\frac{2\alpha-3}{2\alpha}+\mu_\alpha\big)$ and $\big(\frac{3-\alpha}{2\alpha}-\mu_\alpha,\frac12-\mu_\alpha\big)$ is non-empty, we must choose $\frac{3(2-\alpha)}{4\alpha}<\mu_\alpha<\frac14$,
see Remark \ref{re:bound}.

Finally, we prove the trace regularity result: for  any $u_1, \nabla u_2\in D((-\Delta)^{\mu_\alpha})$, $\frac{3(2-\alpha)}{4\alpha}<\mu_\alpha<\frac14$,
we can define the normal derivative $\partial_\nu u$ of $u$ such that  for any $T>0$ we have
\begin{equation*}
\int_0^T\int_{\partial\Omega} \big|\partial_\nu u\big|^2\, d\sigma dt
\le C\big(\lVert u_1\rVert_{D((-\Delta)^{\mu_\alpha})}
+\lVert\nabla u_2\rVert_{D((-\Delta)^{\mu_\alpha})}\big),
\end{equation*}
see Theorem \ref{th:hidalpha}. It is worth noting that the trace result is established under weaker conditions on the initial data compared to the classical assumptions,  see Remark \ref{re:bound1}.

\section{Preliminaries}\label{s:pre}

This section collects some known notations, definitions and results that we will need in the following.

Let $\Omega\subset\mathbb{R}^N$,  $N\ge1$, be a bounded open set with $C^2$ boundary $\partial\Omega$. We consider
$L^2(\Omega)$ endowed with the inner product and norm defined by
\begin{equation*}
\langle u,v\rangle=\int_{\Omega}u(x)v(x)\, dx,
\quad
\lVert u\rVert_{L^2(\Omega)}=\bigg(\int_{\Omega}|u(x)|^{2}\, dx\bigg)^{1/2},
\quad
u,v\in L^2(\Omega).
\end{equation*}
We denote the 
Riemann--Liouville fractional integral operators of order $\beta>0$ by
\begin{equation}\label{eq:RLfi}
I_{0+}^{\beta}u(t)=\frac1{\Gamma(\beta)}\int_0^t (t-s)^{\beta-1}u(s)\, ds,
\end{equation}
\begin{equation}\label{eq:RLfiT}
I_{T-}^{\beta}u(t)=\frac1{\Gamma(\beta)}\int_t^T (s-t)^{\beta-1}u(s)\, ds,
\end{equation}
where $u\in L^2(0,T)$ and $\Gamma (\beta)=\int_0^\infty t^{\beta-1}e^{-t}\, dt$ is the gamma function. The semigroup property of the fractional integral operators 
$I_{0+}^{\beta}$ is given by
\begin{equation}\label{eq:semigroup}
I_{0+}^{\beta}I_{0+}^{\gamma}=I_{0+}^{\beta+\gamma}
\qquad
\beta,\gamma>0.
\end{equation}
The Riemann--Liouville fractional derivative $D_{0^+}^{\alpha}u$ of order $\alpha\in(1,2)$  is defined as follows:
\begin{equation*}
\sideset{}{_{0+}^{\alpha}}{\mathop D}u(t)
=\frac{d^2}{dt^2}I_{0+}^{2-\alpha}u(t)
=\frac1{\Gamma(2-\alpha)}\frac{d^2}{dt^2}\int_0^t (t-s)^{1-\alpha}u(s)\, ds.
\end{equation*}
As $\alpha-1\in(0,1)$ we get
\begin{equation}\label{eq:alpha-1}
\sideset{}{_{0+}^{\alpha-1}}{\mathop D}u=\frac{d}{dt}I_{0+}^{2-\alpha}u,
\end{equation}
hence
\begin{equation}\label{eq:RL-T}
\sideset{}{_{0+}^{\alpha}}{\mathop D}u=\frac{d}{dt}\sideset{}{_{0+}^{\alpha-1}}{\mathop D}u.
\end{equation}
Moreover, since $\alpha-2\in(-1,0)$ we have
\begin{equation}\label{eq:RL-T0}
\sideset{}{_{0+}^{\alpha-2}}{\mathop D}u=I_{0+}^{2-\alpha}u.
\end{equation}
A kind of integration by parts is the well-known property:
\begin{equation}\label{eq:change+-}
\int_0^T I_{0+}^{\alpha}f(t)g(t)\, dt=\int_0^T f(t)I_{T-}^{\alpha}g(t)\, dt,
\end{equation}
(see e.g. \cite{KST}).
The symbol $\Delta$ denotes the Laplace operator as usual.  
We consider the operator $-\Delta$ in $L^2(\Omega)$ with domain 
$D(-\Delta)= H^2(\Omega)\cap H^1_0(\Omega)$.
The fractional powers $(-\Delta)^\theta$ are defined for any $\theta>0$, see e.g. \cite{Pazy} and \cite[Example 4.34]{Lunardi}.
We recall that the spectrum of $-\Delta$ consists of a sequence of positive eigenvalues tending to $+\infty$ and there exists an orthonormal basis $\{e_n\}_{n\in\mathbb{N}}$ of $L^2(\Omega)$ consisting of eigenvectors of $-\Delta$. Moreover, we assume that the eigenvalues are distinct numbers and denote by $\lambda_n$ the eigenvalue with eigenvector $e_n$, that is 
$-\Delta e_n=\lambda_n e_n$. For $\theta>0$ the domain $D((-\Delta)^\theta)$ of $(-\Delta)^\theta$ consists of those functions $u\in L^2(\Omega)$ such that
\begin{equation*}
\sum_{n=1}^\infty \lambda_n^{2\theta} |\langle u,e_n\rangle|^2<+\infty
\end{equation*}
and
\begin{equation*}
(-\Delta)^\theta u=\sum_{n=1}^\infty \lambda_n^{\theta} \langle u,e_n\rangle e_n
\quad
u\in D((-\Delta)^\theta).
\end{equation*}
$D((-\Delta)^\theta)$ is a Hilbert space with the norm given by
\begin{equation}\label{eq:norm-frac}
\lVert u\rVert_{D((-\Delta)^\theta)}=\lVert(-\Delta)^\theta u\rVert_{L^2(\Omega)}
=\bigg(\sum_{n=1}^\infty \lambda_n^{2\theta} |\langle u,e_n\rangle|^2\bigg)^{1/2}
\quad
u\in D((-\Delta)^\theta).
\end{equation}
For any $0<\theta_1<\theta_2$ we have $D((-\Delta)^{\theta_2})\subset D((-\Delta)^{\theta_1})$.

We also have $D((-\Delta)^\theta)\subset H^{2\theta}(\Omega)$ and, in particular, $D((-\Delta)^{\frac12})=H^1_0(\Omega)$.
Identifying the dual $(L^{2}(\Omega))'$ with $L^{2}(\Omega)$ itself we have
$D((-\Delta)^\theta)\subset L^{2}(\Omega)\subset(D((-\Delta)^\theta))'$.
From now on we set
\begin{equation}
D((-\Delta)^{-\theta})=(D((-\Delta)^\theta))',
\end{equation}
hence the elements of $D((-\Delta)^{-\theta})$ are bounded linear functionals on $D((-\Delta)^\theta)$. 
If $\varphi\in D((-\Delta)^{-\theta})$ and $u\in D((-\Delta)^\theta)$ the value $\varphi(u)$ is denoted by
\begin{equation}\label{eq:duality}
\varphi(u)=\langle \varphi,u\rangle_{-\theta,\theta}.
\end{equation}
In addition, $D((-\Delta)^{-\theta})$ is a Hilbert space with the norm given by
\begin{equation}\label{eq:norm-theta}
\lVert\varphi\rVert_{D((-\Delta)^{-\theta})}
=\bigg(\sum_{n=1}^\infty \lambda_n^{-2\theta} |\langle \varphi,e_n\rangle_{-\theta,\theta}|^2\bigg)^{1/2}
\qquad
\varphi\in D((-\Delta)^{-\theta}).
\end{equation}
We also recall that
\begin{equation}\label{eq:-theta}
\langle \varphi,u\rangle_{-\theta,\theta}=\langle\varphi,u\rangle
\qquad \mbox{for}\ \varphi\in L^{2}(\Omega)\,,u\in D((-\Delta)^\theta),
\end{equation}
see e.g., \cite[Chapitre V]{B}.

The Mittag--Leffler function depending on arbitrary constants $\alpha,\beta> 0$ is defined as
\begin{equation}
E_{\alpha,\beta}(z)= \sum_{k=0}^\infty\frac{z^k}{\Gamma(\alpha k+\beta)}
\qquad z\in\mathbb{C}.
\end{equation}
The power series $E_{\alpha,\beta}(z)$ is an entire function
of $z\in\mathbb{C}$. The symbol $E_{\alpha}(z)$ usually denotes
$E_{\alpha,1}(z)$.

\begin{lemma}
Let $\alpha\in(1,2)$ and $\beta>0$. 
\begin{enumerate}[(i)]
\item
There exists a constant $C=C(\alpha,\beta)>0$ such that
\begin{equation}\label{eq:stimeE}
\big\vert E_{\alpha,\beta}(-\mu)\big\vert\le \frac{C}{1+\mu},
\qquad \text{for any}\ \mu\ge0.
\end{equation}
\item
For any $\lambda\in\mathbb{R}$
\begin{equation}\label{eq:I-E}
I_{0+}^{2-\alpha}\big(\tau^{\beta-1}E_{\alpha,\beta}(\lambda \tau^{\alpha})\big)(t)
= t^{1-\alpha+\beta}E_{\alpha,2-\alpha+\beta}(\lambda t^{\alpha}).
\end{equation}

\end{enumerate} 
\end{lemma} 
The proof of (i) can be found in \cite[Theorem1.6]{Pod}. 
For point (ii) see e.g., \cite[formula (2.1.53)]{KST}.

\begin{lemma}
If $\alpha\,,\lambda>0$, then we have
\begin{equation}\label{eq:Ea1}
\frac{d}{dt}E_{\alpha}(-\lambda t^{\alpha})=-\lambda t^{\alpha-1}E_{\alpha,\alpha}(-\lambda t^{\alpha}), 
\qquad t>0,
\end{equation}
\begin{equation}\label{eq:Eaa1}
\frac{d}{dt}\Big(t^{k}E_{\alpha,k+1}(-\lambda t^{\alpha})\Big)=t^{k-1}E_{\alpha,k}(-\lambda t^{\alpha}), 
\qquad k\in\mathbb{N}\,,t\ge0,
\end{equation}
\begin{equation}\label{eq:Eaaa}
\frac{d}{dt}\Big(t^{\alpha-1}E_{\alpha,\alpha}(-\lambda t^{\alpha})\Big)=t^{\alpha-2}E_{\alpha,\alpha-1}(-\lambda t^{\alpha}), 
\qquad t\ge0.
\end{equation}
\end{lemma}
We recall a result about the existence of solutions for fractional scalar differential equations, for the proof see e.g., \cite[Theorem 4.1]{KST}.
\begin{lemma}\label{le:x-y}
For $\alpha\in(1,2)$ and $\lambda,u_1,u_2\in\mathbb{R}$ the solution of the Cauchy type problem
\begin{equation}\label{eq:CP}
\begin{cases}
\displaystyle
\sideset{}{_{0+}^{\alpha}}{\mathop D}u(t)+\lambda u(t)=0,
\quad t\ge0,
\\
\sideset{}{_{0+}^{\alpha-1}}{\mathop D}u(0)=u_1,\quad
\sideset{}{_{0+}^{\alpha-2}}{\mathop D}u(0)=u_2,
\end{cases}
\end{equation}
is given by
\begin{equation*}
u(t)=u_1 t^{\alpha-1}E_{\alpha,\alpha}(-\lambda t^\alpha)
+ u_2 t^{\alpha-2} E_{\alpha,\alpha-1}(-\lambda t^\alpha),
\qquad  t\ge0.
\end{equation*}
\end{lemma}
\begin{remark}\label{re:inicon}
We point out that 
the notation $\sideset{}{_{0+}^{\alpha-k}}{\mathop D}u(0)$, $k=1,2$, in \eqref{eq:CP} means that the limit is taken at almost all points of the right-sided neighborhood $(0,\varepsilon)$, $\varepsilon>0$, of $0$ as follows:
\begin{equation}
\sideset{}{_{0+}^{\alpha-1}}{\mathop D}u(0)=\lim_{t\to0+}\sideset{}{_{0+}^{\alpha-1}}{\mathop D}u(t),
\qquad
\sideset{}{_{0+}^{\alpha-2}}{\mathop D}u(0)=\lim_{t\to0+}I_{0+}^{2-\alpha}u(t),
\end{equation}
see \cite[page 136]{KST}.
\end{remark}
To conclude this section, we mention an elementary result useful for managing estimates.
\begin{lemma}\label{le:maxbeta} 
For any $\beta\in (0,1)$ the function $x\to\frac{x^\beta}{1+x}$ gains its maximum on $[0,+\infty[$ at point $\frac\beta{1-\beta}$ and the maximum value is given by
\begin{equation}\label{eq:maxbeta}
\max_{x\ge0}\frac{x^\beta}{1+x}=\beta^\beta(1-\beta)^{1-\beta},
\qquad\beta\in (0,1).
\end{equation}
\end{lemma} 
The symbol $\sim$ between norms indicates two equivalent norms.
\section{Weak solutions of Riemann--Liouville fractional  equations }\label{s:weak}

To begin with, let us introduce the notion of weak solution.
\begin{definition} 
Let $\alpha\in(1,2)$ and $T>0$.
We define $u$ as a weak solution to the boundary value problem
\begin{equation}\label{eq:weakp}
\begin{aligned}
\sideset{}{_{0+}^{\alpha}}{\mathop D}u(t,x)
&=\Delta u(t,x)
\qquad
(t,x)\in (0,T)\times\Omega,
\\
u(t,x)&=0 
\qquad\qquad (t,x)\in (0,T)\times\partial\Omega,
\end{aligned}
\end{equation}
if  $u\in L^2(0,T;H^1_0(\Omega))$, 
$\sideset{}{_{0+}^{\alpha-1}}{\mathop D}u\in L^2(0,T;L^2(\Omega))$  
and for any $\varphi\in H^1_0(\Omega)$ one has $\int_{\Omega} \sideset{}{_{0+}^{\alpha-1}}{\mathop D}u(t,x)\varphi(x)\, dx\in W^{1,1}(0,T)$ and 
\begin{equation}\label{eq:w-int}
\frac{d}{dt}\int_{\Omega} \sideset{}{_{0+}^{\alpha-1}}{\mathop D}u(t,x)\varphi(x)\, dx
+\int_{\Omega} \nabla u(t,x)\cdot\nabla \varphi(x)\, dx =0,
\qquad \text{a.e.}\ t\in (0,T).
\end{equation}
\end{definition}
\begin{remark}\label{re:RLversusC}
We observe that if $u$ is the weak solution of \eqref{eq:weakp}, then $v=I_{0+}^{2-\alpha}u$ is the weak solution of the Caputo fractional problem
\begin{equation}\label{eq:weakpC}
\begin{aligned}
\sideset{^{C}}{_{0+}^{\alpha}}{\mathop D}v(t,x)
&=\Delta v(t,x)
\qquad
(t,x)\in (0,T)\times\Omega,
\\
v(t,x)&=0 
\qquad\qquad (t,x)\in (0,T)\times\partial\Omega,
\end{aligned}
\end{equation}
where $\sideset{^{C}}{_{0+}^{\alpha}}{\mathop D}$ denotes the Caputo derivative. Indeed,  
$
v_{tt}=\sideset{}{_{0+}^{\alpha}}{\mathop D}u=\Delta u
$
and hence 
$\sideset{^{C}}{_{0+}^{\alpha}}{\mathop D}v(t,x)=I_{0+}^{2-\alpha}v_{tt}
=\Delta I_{0+}^{2-\alpha}u=\Delta v$.

\end{remark}
Keeping in mind Remark \ref{re:inicon} we establish an existence and regularity result in the case $\alpha>\frac32$.
\begin{theorem}\label{th:exist} 
Assume $\alpha>\frac32$.
\begin{enumerate}[(i)]
\item 
For $u_1\in L^2(\Omega)$ and $u_2\in H_0^1(\Omega)$  the function
\begin{multline}\label{eq:def-u0}
u(t,x)
\\
=\sum_{n=1}^\infty\big[ \langle u_1,e_n\rangle t^{\alpha-1}E_{\alpha,\alpha}(-\lambda_nt^\alpha)
+\langle u_2,e_n\rangle t^{\alpha-2} E_{\alpha,\alpha-1}(-\lambda_nt^\alpha)\big]e_n(x)
\end{multline}
is the unique weak solution of \eqref{eq:weakp} written as series and satisfying the  initial data 
\begin{equation} \label{eq:indata}
\sideset{}{_{0+}^{\alpha-1}}{\mathop D}u(0,\cdot)=u_{1},\qquad \sideset{}{_{0+}^{\alpha-2}}{\mathop D}u(0,\cdot)=u_{2}.
\end{equation}
Moreover
\begin{multline}\label{eq:def-u-alpha-1}
\sideset{}{_{0+}^{\alpha-1}}{\mathop D}u(t,x)
\\
=\sum_{n=1}^\infty\big[\langle u_1,e_n\rangle E_{\alpha}(-\lambda_nt^\alpha)
-\lambda_n\langle u_2,e_n\rangle t^{\alpha-1} E_{\alpha,\alpha}(-\lambda_nt^\alpha)\big]e_n(x),
\end{multline}
$\sideset{}{_{0+}^{\alpha-1}}{\mathop D}u\in C([0,T];D(A^{-\theta}))$,  $\theta\in\big(\frac{2-\alpha}{2\alpha},\frac12\big)$.
\item 
If $u_1\in H^1_0(\Omega)$ and $u_2\in H^2(\Omega)\cap H^1_0(\Omega)$, then $u\in L^2(0,T;H^2(\Omega))$, 
$\sideset{}{_{0+}^{\alpha}}{\mathop D}u(t,\cdot)\in L^2(0,T;L^2(\Omega))$ is given by
\begin{multline}\label{eq:def-u-alpha}
\sideset{}{_{0+}^{\alpha}}{\mathop D}u(t,x)
\\
=-\sum_{n=1}^\infty\lambda_n\big[ \langle u_1,e_n\rangle t^{\alpha-1}E_{\alpha,\alpha}(-\lambda_nt^\alpha)
+\langle u_2,e_n\rangle t^{\alpha-2} E_{\alpha,\alpha-1}(-\lambda_nt^\alpha)\big]e_n(x),
\end{multline}
and
\begin{equation}\label{eq:def-u-alpha1}
\sideset{}{_{0+}^{\alpha}}{\mathop D}u(t,x)
=\Delta u(t,x)
\quad
\text{a.e.}\ (t,x)\in (0,T)\times\Omega.
\end{equation}

\end{enumerate}

\end{theorem}

\begin{proof} 
(i) First, we show that 
\begin{equation*}
\sum_{n=1}^\infty\big[ \langle u_1,e_n\rangle t^{\alpha-1}E_{\alpha,\alpha}(-\lambda_nt^\alpha)
+\langle u_2,e_n\rangle t^{\alpha-2} E_{\alpha,\alpha-1}(-\lambda_nt^\alpha)\big]e_n(x)
\end{equation*}
is the unique series 
giving the weak solution of \eqref{eq:weakp} with  initial data 
\eqref{eq:indata}.
For this purpose
we seek the solution in the form 
$
u(t)=\sum_{n=1}^\infty\ u_n(t)e_n
$
where the functions 
$u_n(t)=\langle u(t),e_n\rangle$ 
are unknown. We take $\varphi=e_n$ in \eqref{eq:w-int} to get
\begin{equation*}
\frac{d}{dt}\int_{\Omega} \sideset{}{_{0+}^{\alpha-1}}{\mathop D}u(t,x)e_n(x)\, dx
-\int_{\Omega} u(t,x)\Delta e_n(x)\, dx =0,
\qquad t\in [0,T],
\end{equation*}
hence by \eqref{eq:RL-T} we have
\begin{equation*}
\sideset{}{_{0+}^{\alpha}}{\mathop D}u_n+\lambda_n u_n=
\frac{d}{dt}\sideset{}{_{0+}^{\alpha-1}}{\mathop D}u_n+\lambda_n u_n=0.
\end{equation*}
So $u_n(t)$ is the solution of  problem \eqref{eq:CP} with $\lambda=\lambda_n$ where, thanks to \eqref{eq:indata}, the initial conditions are given by
\begin{equation*}
\sideset{}{_{0+}^{\alpha-1}}{\mathop D}u_n(0,\cdot)=\langle u_1,e_n\rangle,\quad
\sideset{}{_{0+}^{\alpha-2}}{\mathop D}u_n(0,\cdot)=\langle u_2,e_n\rangle.
\end{equation*}
Therefore by Lemma \ref{le:x-y} we get
\begin{equation}\label{eq:u_n}
u_n(t)=\langle u_1,e_n\rangle t^{\alpha-1}E_{\alpha,\alpha}(-\lambda_nt^\alpha)
+\langle u_2,e_n\rangle t^{\alpha-2} E_{\alpha,\alpha-1}(-\lambda_nt^\alpha),
\qquad t\ge0,
\end{equation}
hence the uniqueness follows.

Now, we show that, taking $u_1\in L^2(\Omega)$ and $u_2\in H_0^1(\Omega)$, 
$
u(t,\cdot)=\sum_{n=1}^\infty u_n(t)e_n
$,
with $u_n(t)$ given by \eqref{eq:u_n}, is a weak solution of \eqref{eq:weakp} satisfying the initial conditions \eqref{eq:indata}. First, we note that $ u(t,\cdot)\in H^1_0(\Omega)$ for $t\in (0,T]$. Indeed
\begin{multline*}
\lVert u(t,\cdot)\rVert_{H^1_0(\Omega)}^2
=\sum_{n=1}^\infty \lambda_n | u_n(t)|^2
\\
\le
2\sum_{n=1}^\infty \lambda_n \big|\langle u_1,e_n\rangle t^{\alpha-1}E_{\alpha,\alpha}(-\lambda_nt^\alpha) \big|^2
+2\sum_{n=1}^\infty \lambda_n \big| \langle u_2,e_n\rangle t^{\alpha-2} E_{\alpha,\alpha-1}(-\lambda_nt^\alpha)\big|^2
\end{multline*}
and thanks to \eqref{eq:stimeE}  
we have
\begin{align*}
 \lambda_n \big|\langle u_1,e_n\rangle t^{\alpha-1}E_{\alpha,\alpha}(-\lambda_nt^\alpha)\big|^2
& \le C t^{\alpha-2} \big|\langle u_1,e_n\rangle \big|^2\frac{\lambda_nt^\alpha}{(1+\lambda_nt^\alpha)^2}
\le C t^{\alpha-2} \big|\langle u_1,e_n\rangle \big|^2,
\\
 \lambda_n \big| \langle u_2,e_n\rangle t^{\alpha-2} E_{\alpha,\alpha-1}(-\lambda_nt^\alpha) \big|^2
 &\le Ct^{2\alpha-4}\lambda_n \big| \langle u_2,e_n\rangle  \big|^2,
\end{align*}
hence
\begin{equation}\label{eq:deltau}
\lVert u(t,\cdot)\rVert_{H^1_0(\Omega)}^2
\le C t^{\alpha-2}\|u_1\|_{L^2(\Omega)}^2+C t^{2\alpha-4} \| u_2\|_{H^1_0(\Omega)}^2.
\end{equation}
As $\alpha>\frac32$ we obtain $u\in L^2(0,T;H^1_0(\Omega))$ and
\begin{equation*}
\int_0^T\lVert u(t,\cdot)\rVert_{H^1_0(\Omega)}^2\, dt
\le C T^{\alpha-1}\|u_1\|_{L^2(\Omega)}^2+C T^{2\alpha-3} \| u_2\|_{H^1_0(\Omega)}^2.
\end{equation*}
Moreover, thanks  to \eqref{eq:I-E} we have 
\begin{equation}\label{eq:Iun}
I_{0+}^{2-\alpha}u_n(t)
=
 \langle u_1,e_n\rangle t E_{\alpha,2}(-\lambda_nt^\alpha)
+\langle u_2,e_n\rangle  E_{\alpha}(-\lambda_nt^\alpha),
\end{equation}
hence, being $u\in L^2(0,T;H^1_0(\Omega))$, by the dominated convergence theorem we get
\begin{equation}\label{eq:}
I_{0+}^{2-\alpha}u(t,\cdot)
=
\sum_{n=1}^\infty\big[ \langle u_1,e_n\rangle t E_{\alpha,2}(-\lambda_nt^\alpha)
+\langle u_2,e_n\rangle  E_{\alpha}(-\lambda_nt^\alpha)\big]e_n.
\end{equation}
The series 
$\sum_{n=1}^\infty I_{0+}^{2-\alpha}u_n(t)e_n$ 
is convergent in $L^2(\Omega)$ uniformly in $t\in [0,T]$. As a consequence we have
$I_{0+}^{2-\alpha}u\in C([0,T];L^2(\Omega))$
and
\begin{equation*}
\sideset{}{_{0+}^{\alpha-2}}{\mathop D}u(0,\cdot)=\lim_{t\to0+}I_{0+}^{2-\alpha}u(t,\cdot)
=\sum_{n=1}^\infty \langle u_2,e_n\rangle e_n=u_2,
\end{equation*}
taking into account that $E_{\alpha,2}(0)=\frac1{\Gamma(2)}$ and $E_{\alpha}(0)=\frac1{\Gamma(1)}=1$.

We have to show that $\sideset{}{_{0+}^{\alpha-1}}{\mathop D}u$ is given by \eqref{eq:def-u-alpha-1} and belongs to
$C([0,T];D(A^{-\theta}))$,
for $\theta\in\big(\frac{2-\alpha}{2\alpha},\frac12\big)$. First we note that, thanks to \eqref{eq:Iun}, \eqref{eq:Eaa1} for $k=1$ and \eqref{eq:Ea1}, we get
\begin{equation*}
\frac{d}{dt}I_{0+}^{2-\alpha}u_n(t)
=\langle u_1,e_n\rangle  E_{\alpha}(-\lambda_nt^\alpha)
-\langle u_2,e_n\rangle \lambda_n t^{\alpha-1} E_{\alpha,\alpha}(-\lambda_nt^\alpha).
\end{equation*}
Since
\begin{equation*}
\Big\|\sum_{n=1}^\infty \frac{d}{dt}I_{0+}^{2-\alpha}u_n(t)e_n\Big\|_{D(A^{-\theta})}^2
=\sum_{n=1}^\infty\lambda_n^{-2\theta}\Big|\frac{d}{dt}I_{0+}^{2-\alpha}u_n(t)\Big|^2,
\end{equation*}
due to \eqref{eq:stimeE}, \eqref{eq:maxbeta} and $0<\theta<\frac12$ we obtain
\begin{multline*}
\lambda_n^{-2\theta}\big|\langle u_2,e_n\rangle \lambda_n t^{\alpha-1}E_{\alpha,\alpha}(-\lambda_nt^\alpha)\big|^2
\\
\le Ct^{2\alpha\theta+\alpha-2} \lambda_n |\langle u_2,e_n\rangle \big|^2\bigg(\frac{(\lambda_n t^{\alpha})^{\frac12-\theta}}{1+\lambda_nt^\alpha}\bigg)^2
\le Ct^{2\alpha\theta+\alpha-2} \lambda_n |\langle u_2,e_n\rangle \big|^2.
\end{multline*}
Therefore, taking into account that $\theta>\frac{2-\alpha}{2\alpha}$, for any $t\in [0,T]$ we get
\begin{equation*}
\Big\|\sum_{n=1}^\infty \frac{d}{dt}I_{0+}^{2-\alpha}u_n(t)e_n\Big\|_{D(A^{-\theta})}^2
\le
C\|u_1\|_{L^2(\Omega)}^2+CT^{2\alpha\theta+\alpha-2} \| u_2\|_{H^1_0(\Omega)}^2.
\end{equation*}
Furthermore, the series $\sum_{n=1}^\infty \frac{d}{dt}I_{0+}^{2-\alpha}u_n(t)e_n$
is convergent in $D(A^{-\theta})$ uniformly in $t\in [0,T]$.
For that reason, the function $I_{0+}^{2-\alpha}u$ is differentiable, so  
keeping in mind \eqref{eq:alpha-1} we  get
\begin{multline*}
\sideset{}{_{0+}^{\alpha-1}}{\mathop D}u(t,\cdot)=\frac{d}{dt}I_{0+}^{2-\alpha}u(t,\cdot)
\\
=
\sum_{n=1}^\infty\big[ \langle u_1,e_n\rangle  E_{\alpha}(-\lambda_nt^\alpha)
-\langle u_2,e_n\rangle \lambda_n t^{\alpha-1} E_{\alpha,\alpha}(-\lambda_nt^\alpha)\big]e_n,
\end{multline*}
that is \eqref{eq:def-u-alpha-1}.
Since $\sideset{}{_{0+}^{\alpha-1}}{\mathop D}u\in C([0,T];D(A^{-\theta}))$, from \eqref{eq:def-u-alpha-1} it follows
\begin{equation*}
\sideset{}{_{0+}^{\alpha-1}}{\mathop D}u(0,\cdot)=\lim_{t\to0+}\sideset{}{_{0+}^{\alpha-1}}{\mathop D}u(t,\cdot)
=\sum_{n=1}^\infty \langle u_1,e_n\rangle e_n=u_1.
\end{equation*}
Using again \eqref{eq:stimeE}  
we obtain
\begin{equation*}
\lVert \sideset{}{_{0+}^{\alpha-1}}{\mathop D}u(t,\cdot)\rVert_{L^2(\Omega)}^2
\le C \|u_1\|_{L^2(\Omega)}^2+C t^{\alpha-2} \| u_2\|_{H^1_0(\Omega)}^2,
\end{equation*}
and hence 
\begin{equation*}
\int_0^T\lVert \sideset{}{_{0+}^{\alpha-1}}{\mathop D}u(t,\cdot)\rVert_{L^2(\Omega)}^2\, dt
\le C \|u_1\|_{L^2(\Omega)}^2+C T^{\alpha-1} \| u_2\|_{H^1_0(\Omega)}^2,
\end{equation*}
that is $\sideset{}{_{0+}^{\alpha-1}}{\mathop D}u\in L^2(0,T;L^2(\Omega))$. 

Afterwards, if $\varphi=\sum_{n=1}^\infty\ \langle\varphi,e_n\rangle e_n$ belongs to $H^1_0(\Omega)$, then due to \eqref{eq:def-u-alpha-1}
we have
\begin{multline*}
\int_{\Omega} \sideset{}{_{0+}^{\alpha-1}}{\mathop D}u(t,x)\varphi(x)\, dx
\\
=
\sum_{n=1}^\infty\big[\langle u_1,e_n\rangle E_{\alpha}(-\lambda_nt^\alpha)
-\lambda_n\langle u_2,e_n\rangle t^{\alpha-1} E_{\alpha,\alpha}(-\lambda_nt^\alpha)\big] 
\langle\varphi,e_n\rangle.
\end{multline*}
We observe that \eqref{eq:Ea1} and \eqref{eq:Eaaa} yield
\begin{multline}\label{eq:absc0}
\frac{d}{dt}\langle \sideset{}{_{0+}^{\alpha-1}}{\mathop D}u(t,\cdot),e_n\rangle
\\
=-\lambda_n\big[ \langle u_1,e_n\rangle t^{\alpha-1} E_{\alpha,\alpha}(-\lambda_nt^\alpha)
+ \langle u_2,e_n\rangle t^{\alpha-2} E_{\alpha,\alpha-1}(-\lambda_nt^\alpha)\big],
\end{multline}
and hence
\begin{multline}\label{eq:absc}
\sum_{n=1}^\infty \frac{d}{dt}\langle \sideset{}{_{0+}^{\alpha-1}}{\mathop D}u(t,\cdot),e_n\rangle\langle \varphi,e_n\rangle
\\
=-\sum_{n=1}^\infty\lambda_n^{\frac12}\big[ \langle u_1,e_n\rangle t^{\alpha-1} E_{\alpha,\alpha}(-\lambda_nt^\alpha)
+ \langle u_2,e_n\rangle t^{\alpha-2} E_{\alpha,\alpha-1}(-\lambda_nt^\alpha)\big] \lambda_n^{\frac12}\langle \varphi,e_n\rangle.
\end{multline}
Thanks to  \eqref{eq:stimeE} and \eqref{eq:maxbeta} we get
\begin{multline*}
\sum_{n=1}^\infty\lambda_n\big|  \langle u_1,e_n\rangle t^{\alpha-1} E_{\alpha,\alpha}(-\lambda_nt^\alpha)
+ \langle u_2,e_n\rangle t^{\alpha-2} E_{\alpha,\alpha-1}(-\lambda_nt^\alpha)\big|^2 
\\
\le
Ct^{\alpha-2}\sum_{n=1}^\infty\big| \langle u_1,e_n\rangle\big|^2\Big(\frac{ \lambda_n^{\frac12}t^{\frac\alpha2}}{ 1+\lambda_nt^\alpha}\Big)^2 
+Ct^{2\alpha-4}\sum_{n=1}^\infty\lambda_n \big|\langle u_2,e_n\rangle\big|^2
\\
\le
C t^{\alpha-2}\|u_1\|_{L^2(\Omega)}^2+C t^{2\alpha-4} \| u_2\|_{H^1_0(\Omega)}^2.
\end{multline*}
Since $\alpha>\frac32$, from \eqref{eq:absc} we deduce
\begin{multline*}
\int_0^T\Big|\sum_{n=1}^\infty \frac{d}{dt}\langle \sideset{}{_{0+}^{\alpha-1}}{\mathop D}u(t,\cdot),e_n\rangle
\langle \varphi,e_n\rangle\Big|\, dt
\\
\le
C T^{\alpha-1}\|u_1\|_{L^2(\Omega)}^2+C T^{2\alpha-3} \| u_2\|_{H^1_0(\Omega)}^2+C\|\varphi\|_{H^1_0(\Omega)}^2,
\end{multline*}
that is $\sum_{n=1}^\infty \frac{d}{dt}\langle \sideset{}{_{0+}^{\alpha-1}}{\mathop D}u(t,\cdot),e_n\rangle  \langle \varphi,e_n\rangle$
belongs to $L^1(0,T)$.
Therefore, for any $t\in [0,T]$ we have
\begin{multline*}
\int_0^t\sum_{n=1}^\infty \frac{d}{ds}\langle \sideset{}{_{0+}^{\alpha-1}}{\mathop D}u(s,\cdot),e_n\rangle\langle \varphi,e_n\rangle\, ds
\\
=\sum_{n=1}^\infty\int_0^t \frac{d}{ds}\langle \sideset{}{_{0+}^{\alpha-1}}{\mathop D}u(s,\cdot),e_n\rangle
\langle \varphi,e_n\rangle\, ds
\\
=
\int_{\Omega} \sideset{}{_{0+}^{\alpha-1}}{\mathop D}u(t,x)\varphi(x)\, dx
-\int_{\Omega}u_1(x)\varphi(x)\, dx.
\end{multline*}
As a consequence the function $t\to\int_{\Omega} \sideset{}{_{0+}^{\alpha-1}}{\mathop D}u(t,x)\varphi(x)\, dx$ is absolutely continuous and,
thanks to \eqref{eq:absc}, its $L^1-$ derivative is given by
\begin{multline*}
 \frac{d}{dt}\int_{\Omega} \sideset{}{_{0+}^{\alpha-1}}{\mathop D}u(t,x)\varphi(x)\, dx
\\
=-\sum_{n=1}^\infty\lambda_n^{\frac12}\big[ \langle u_1,e_n\rangle t^{\alpha-1} E_{\alpha,\alpha}(-\lambda_nt^\alpha)
+ \langle u_2,e_n\rangle t^{\alpha-2} E_{\alpha,\alpha-1}(-\lambda_nt^\alpha)\big] \lambda_n^{\frac12}\langle \varphi,e_n\rangle.
\end{multline*}
On the other hand by \eqref{eq:def-u0} we get
\begin{multline*}
\int_{\Omega} \nabla u(t,x)\cdot\nabla \varphi(x)\, dx
\\
=\sum_{n=1}^\infty\lambda_n^{\frac12}\big[ \langle u_1,e_n\rangle t^{\alpha-1} E_{\alpha,\alpha}(-\lambda_nt^\alpha)
+ \langle u_2,e_n\rangle t^{\alpha-2} E_{\alpha,\alpha-1}(-\lambda_nt^\alpha)\big] \lambda_n^{\frac12}\langle \varphi,e_n\rangle,
\end{multline*}
and hence for a.e. $t\in (0,T)$ we have
\begin{equation*}
\frac{d}{dt}\int_{\Omega} \sideset{}{_{0+}^{\alpha-1}}{\mathop D}u(t,x)\varphi(x)\, dx
+\int_{\Omega} \nabla u(t,x)\cdot\nabla \varphi(x)\, dx =0,
\end{equation*}
that is \eqref{eq:w-int} holds. 

In conclusion, $u$ given by \eqref{eq:def-u0} is the weak solution of \eqref{eq:weakp} satisfying the  initial data 
\eqref{eq:indata}.

(ii) We assume $u_1\in H^1_0(\Omega)$ and $u_2\in H^2(\Omega)\cap H^1_0(\Omega)$. We prove the regularity results about the weak solution $u$ arguing as before and using again \eqref{eq:stimeE} and \eqref{eq:maxbeta}. First, we obtain
\begin{equation*}
\int_0^T\lVert u(t,\cdot)\rVert_{H^2(\Omega)}^2\, dt
\le C T^{\alpha-1}\|u_1\|_{H^1_0(\Omega)}^2+C T^{2\alpha-3} \|u_2\|_{H^2(\Omega)}^2,
\end{equation*}
that is $u\in L^2(0,T;H^2(\Omega))$.

To prove the regularity of $\sideset{}{_{0+}^{\alpha}}{\mathop D}u$, taking into account \eqref{eq:RL-T}, we have to differentiate formula 
\eqref{eq:def-u-alpha-1}. Indeed, from \eqref{eq:absc0} we deduce
\begin{multline}\label{eq:l2der}
\int_0^T\lVert\sum_{n=1}^\infty \frac{d}{dt}\langle \sideset{}{_{0+}^{\alpha-1}}{\mathop D}u(t,\cdot),e_n\rangle e_n\rVert_{L^2(\Omega)}^2\, dt
\\
\le
C T^{\alpha-1}\|u_1\|_{H^1_0(\Omega)}^2+C T^{2\alpha-3} \| u_2\|_{H^2(\Omega)}^2.
\end{multline}
Therefore,  the function $\sideset{}{_{0+}^{\alpha-1}}{\mathop D}u(t,\cdot)$ is absolutely continuous and, also because of \eqref{eq:absc0}, its derivative is given by
\begin{multline*}
 \frac{d}{dt}\sideset{}{_{0+}^{\alpha-1}}{\mathop D}u(t,\cdot)
\\
=-\sum_{n=1}^\infty\lambda_n\big[ \langle u_1,e_n\rangle t^{\alpha-1} E_{\alpha,\alpha}(-\lambda_nt^\alpha)
+ \langle u_2,e_n\rangle t^{\alpha-2} E_{\alpha,\alpha-1}(-\lambda_nt^\alpha)\big]e_n.
\end{multline*}
Since
$\sideset{}{_{0+}^{\alpha}}{\mathop D}u=\frac{d}{dt}\sideset{}{_{0+}^{\alpha-1}}{\mathop D}u$ the formula \eqref{eq:def-u-alpha} is true.
Finally, the inequality \eqref{eq:l2der} guarantees 
$\sideset{}{_{0+}^{\alpha}}{\mathop D}u\in L^2(0,T;L^2(\Omega))$ and \eqref{eq:def-u-alpha1} follows from \eqref{eq:w-int} and the regularity of the solution $u$.
\end{proof} 

\section{Regularity in the interpolation spaces when $\alpha>\frac32$ }\label{s:reg}


We establish a result on the regularity of weak solutions by assuming that the data $u_1$ and $u_2$ belong to interpolation spaces connected to the Laplace operator.
 
%

\begin{theorem}\label{th:reg-l2} 
Assume $\alpha>\frac32$.
For $u_1, \nabla u_2\in D((-\Delta)^{\mu_\alpha})$, $\mu_\alpha\ge0$,  the unique weak solution $u$ to \eqref{eq:weakp}--\eqref{eq:indata} given by \eqref{eq:def-u0} satisfies the following
\begin{enumerate}[(i)]
\item 
for $\theta\in\big(\mu_\alpha,\frac{2\alpha-3}{2\alpha}+\mu_\alpha\big)$
\begin{equation}\label{eq:nabla-l2}
\lVert\nabla u\rVert_{L^2(0,T;D((-\Delta)^{\theta}))}\le C\big(\lVert u_1\rVert_{D((-\Delta)^{\mu_\alpha})}+\lVert\nabla u_2\rVert_{D((-\Delta)^{\mu_\alpha})}\big),
\end{equation}
\item 
for $\theta\in\big(\frac{3-\alpha}{2\alpha}-\mu_\alpha,\frac12-\mu_\alpha\big)$ 
\begin{equation}\label{eq:partial-l2}
\lVert\sideset{}{_{0+}^{\alpha}}{\mathop D}u\rVert_{L^2(0,T;D((-\Delta)^{-\theta}))}\le C\big(\lVert u_1\rVert_{D((-\Delta)^{\mu_\alpha})}
+\lVert\nabla u_2\rVert_{D((-\Delta)^{\mu_\alpha})}\big),
\end{equation}
\end{enumerate}
for some constant $C>0$.

\end{theorem}

\begin{proof} 
(i) In virtue of the expression \eqref{eq:def-u0} for the solution $u$ 
we have
\begin{multline}\label{eq:nablath}
\lVert\nabla u(t,\cdot)\rVert_{D((-\Delta)^{\theta})}^2
\\
=\sum_{n=1}^\infty\lambda_n^{1+2\theta}\big| \langle u_1,e_n\rangle t^{\alpha-1}E_{\alpha,\alpha}(-\lambda_nt^\alpha)
+\langle u_2,e_n\rangle t^{\alpha-2} E_{\alpha,\alpha-1}(-\lambda_nt^\alpha)\big|^2
\\
\le
2t^{2\alpha-2}\sum_{n=1}^\infty\lambda_n^{1+2\theta}\big| \langle u_1,e_n\rangle E_{\alpha,\alpha}(-\lambda_nt^\alpha)\big|^2
\\
+2t^{2\alpha-4}\sum_{n=1}^\infty\lambda_n^{1+2\theta}\big| \langle u_2,e_n\rangle E_{\alpha,\alpha-1}(-\lambda_nt^\alpha)\big|^2.
\end{multline}
To estimate the first sum we use \eqref{eq:stimeE} and keep in mind that $u_1\in D((-\Delta)^{\mu_{\alpha}})$, so we get
\begin{equation*}
\sum_{n=1}^\infty\lambda_n^{1+2\theta}\big| \langle u_1,e_n\rangle E_{\alpha,\alpha}(-\lambda_nt^\alpha)\big|^2
\le
C\sum_{n=1}^\infty\lambda_n^{2\mu_{\alpha}}|\langle u_1,e_n\rangle|^2 \frac{\lambda_n^{1+2(\theta-\mu_{\alpha})}}{(1+\lambda_nt^\alpha)^2}.
\end{equation*}
Since
\begin{equation*}
\frac{\lambda_n^{1+2(\theta-\mu_{\alpha})}}{(1+\lambda_nt^\alpha)^2}
 =\Big(\frac{(\lambda_nt^\alpha)^{\frac12+\theta-\mu_{\alpha}}}{1+\lambda_nt^\alpha}\Big)^2 t^{-\alpha-2\alpha(\theta-\mu_{\alpha})},
\end{equation*}
assuming $\mu_{\alpha}<\theta<\frac12+\mu_{\alpha}$, we  can apply \eqref{eq:maxbeta} to have
\begin{multline}\label{eq:u1}
\sum_{n=1}^\infty\lambda_n^{1+2\theta}\big| \langle u_1,e_n\rangle E_{\alpha,\alpha}(-\lambda_nt^\alpha)\big|^2
\\
\le
Ct^{-\alpha-2\alpha(\theta-\mu_{\alpha})}\sum_{n=1}^\infty\lambda_n^{2\mu_{\alpha}}|\langle u_1,e_n\rangle|^2 
=Ct^{-\alpha-2\alpha(\theta-\mu_{\alpha})}\lVert u_1\rVert_{D((-\Delta)^{\mu_{\alpha}})}^2.
\end{multline}
To evaluate the second sum in \eqref{eq:nablath} we again use  \eqref{eq:stimeE}. Bearing in mind that 
$\nabla u_2\in D((-\Delta)^{\mu_{\alpha}})$ we have
\begin{equation*}
\sum_{n=1}^\infty\lambda_n^{1+2\theta}\big| \langle u_2,e_n\rangle E_{\alpha,\alpha-1}(-\lambda_nt^\alpha)\big|^2
\le
C\sum_{n=1}^\infty\lambda_n^{1+2\mu_{\alpha}}|\langle u_2,e_n\rangle|^2 \frac{\lambda_n^{2(\theta-\mu_{\alpha})}}{(1+\lambda_nt^\alpha)^2}.
\end{equation*}
We note that
\begin{equation*}
\frac{\lambda_n^{2(\theta-\mu_{\alpha})}}{(1+\lambda_nt^\alpha)^2}
 =\Big(\frac{(\lambda_nt^\alpha)^{\theta-\mu_{\alpha}}}{1+\lambda_nt^\alpha}\Big)^2 t^{-2\alpha(\theta-\mu_{\alpha})}.
\end{equation*}
Thanks to $\mu_{\alpha}<\theta<\frac12+\mu_{\alpha}$ and \eqref{eq:maxbeta}, we obtain
\begin{equation*}
\sum_{n=1}^\infty\lambda_n^{1+2\theta}\big| \langle u_2,e_n\rangle E_{\alpha,\alpha-1}(-\lambda_nt^\alpha)\big|^2
\le
Ct^{-2\alpha(\theta-\mu_{\alpha})}\lVert\nabla u_2\rVert_{D((-\Delta)^{\mu_{\alpha}})}^2.
\end{equation*}
Plugging the previous estimate and \eqref{eq:u1}  into \eqref{eq:nablath} yields
\begin{equation*}
\lVert\nabla u(t,\cdot)\rVert_{D((-\Delta)^{\theta})}^2
\le
Ct^{\alpha-2-2\alpha(\theta-\mu_{\alpha})}\lVert u_1\rVert_{D((-\Delta)^{\mu_{\alpha}})}^2+
Ct^{2\alpha-4-2\alpha(\theta-\mu_{\alpha})}\lVert\nabla u_2\rVert_{D((-\Delta)^{\mu_{\alpha}})}^2.
\end{equation*}
For $\alpha-1-2\alpha(\theta-\mu_{\alpha})>0$ and
$2\alpha-3-2\alpha(\theta-\mu_{\alpha})>0$ we obtain that $\nabla u$ belongs to $L^2(0,T;D((-\Delta)^{\theta}))$. Since $\alpha<2$ we have 
$\frac{2\alpha-3}{2\alpha}<\frac{\alpha-1}{2\alpha}$, and hence
we must take $\theta<\frac{2\alpha-3}{2\alpha}+\mu_{\alpha}$, which is consistent with the previous condition $\theta<\frac12+\mu_{\alpha}$.
In conclusion, \eqref{eq:nabla-l2} follows.

(ii) In proving the estimate, we follow the same lines of reasoning used above. 

Thanks to \eqref{eq:def-u-alpha} one has
\begin{multline}\label{eq:par-u-alpha}
\lVert\sideset{}{_{0+}^{\alpha}}{\mathop D}u(t,\cdot)\rVert_{D((-\Delta)^{-\theta})}^2
\\
=\sum_{n=1}^\infty\lambda_n^{2-2\theta}\big| \langle u_1,e_n\rangle t^{\alpha-1}E_{\alpha,\alpha}(-\lambda_nt^\alpha)
+\langle u_2,e_n\rangle t^{\alpha-2} E_{\alpha,\alpha-1}(-\lambda_nt^\alpha)\big|^2
\\
\le
2t^{2\alpha-2}\sum_{n=1}^\infty\lambda_n^{2-2\theta}\big| \langle u_1,e_n\rangle E_{\alpha,\alpha}(-\lambda_nt^\alpha)\big|^2
\\
+2 t^{2\alpha-4}\sum_{n=1}^\infty\lambda_n^{2-2\theta}\big| \langle u_2,e_n\rangle E_{\alpha,\alpha-1}(-\lambda_nt^\alpha)\big|^2.
\end{multline}
The first sum can be evaluated as follows 
\begin{equation*}
\sum_{n=1}^\infty\lambda_n^{2-2\theta}\big| \langle u_1,e_n\rangle E_{\alpha,\alpha}(-\lambda_nt^\alpha)\big|^2
\le
C\sum_{n=1}^\infty\lambda_n^{2\mu_{\alpha}}\big| \langle u_1,e_n\rangle \big|^2
\frac{\lambda_n^{2-2(\theta+\mu_{\alpha})}}{(1+\lambda_nt^\alpha)^2}.
\end{equation*}
Since
\begin{equation*}
 \frac{\lambda_n^{2-2(\theta+\mu_{\alpha})}}{(1+\lambda_nt^\alpha)^2}
 =\Big(\frac{(\lambda_nt^\alpha)^{1-(\theta+\mu_{\alpha})}}{1+\lambda_nt^\alpha}\Big)^2 t^{2\alpha(\theta+\mu_{\alpha})-2\alpha},
\end{equation*}
we have
\begin{equation}\label{eq:par-u-alpha1}
\sum_{n=1}^\infty\lambda_n^{2-2\theta}\big| \langle u_1,e_n\rangle E_{\alpha,\alpha}(-\lambda_nt^\alpha)\big|^2
\le
C t^{2\alpha(\theta+\mu_{\alpha})-2\alpha}\lVert u_1\rVert_{D((-\Delta)^{\mu_{\alpha}})}^2.
\end{equation}
Regarding the second sum in \eqref{eq:par-u-alpha}, we note that
\begin{equation*}
\sum_{n=1}^\infty\lambda_n^{2-2\theta}\big|\langle u_2,e_n\rangle  E_{\alpha,\alpha-1}(-\lambda_nt^\alpha)\big|^2
\le
C\sum_{n=1}^\infty\lambda_n^{1+2\mu_{\alpha}}|\langle u_2,e_n\rangle|^2 
\frac{\lambda_n^{1-2(\theta+\mu_{\alpha})}}{(1+\lambda_nt^\alpha)^2}
\end{equation*}
and
\begin{equation*}
 \frac{\lambda_n^{1-2(\theta+\mu_{\alpha})}}{(1+\lambda_nt^\alpha)^2}
 =\Big(\frac{(\lambda_nt^\alpha)^{\frac12-(\theta+\mu_{\alpha})}}{1+\lambda_nt^\alpha}\Big)^2 t^{2\alpha(\theta+\mu_{\alpha})-\alpha}.
\end{equation*}
Taking $\theta<\frac12-\mu_{\alpha}$ we get
\begin{equation*}
\sum_{n=1}^\infty\lambda_n^{2-2\theta}\big| \langle u_2,e_n\rangle E_{\alpha,\alpha-1}(-\lambda_nt^\alpha)\big|^2
\le C t^{2\alpha(\theta+\mu_{\alpha})-\alpha}\lVert\nabla u_2\rVert_{D((-\Delta)^{\mu_{\alpha}})}^2.
\end{equation*}
By inserting the previous estimate and \eqref{eq:par-u-alpha1} into \eqref{eq:par-u-alpha}, we obtain
\begin{multline*}
\lVert\sideset{}{_{0+}^{\alpha}}{\mathop D}u(t,\cdot)\rVert_{D((-\Delta)^{-\theta})}^2
\\
\le
C t^{2\alpha(\theta+\mu_{\alpha})-2}\lVert u_1\rVert_{D((-\Delta)^{\mu_{\alpha}})}^2
+C t^{2\alpha(\theta+\mu_{\alpha})+\alpha-4}\lVert\nabla u_2\rVert_{D((-\Delta)^{\mu_{\alpha}})}^2.
\end{multline*}
So, $\sideset{}{_{0+}^{\alpha}}{\mathop D}u(t,\cdot)\in L^2(0,T;D((-\Delta)^{-\theta}))$ if $2\alpha(\theta+\mu_{\alpha})-1>0$ and
$2\alpha(\theta+\mu_{\alpha})+\alpha-3>0$, that is $\theta>\frac{3-\alpha}{2\alpha}-\mu_{\alpha}$. In conclusion, \eqref{eq:partial-l2} is proved.
\end{proof} 

\begin{figure}[t]
          \subfloat[$\xi=0.7$]{\includegraphics[width=7cm]{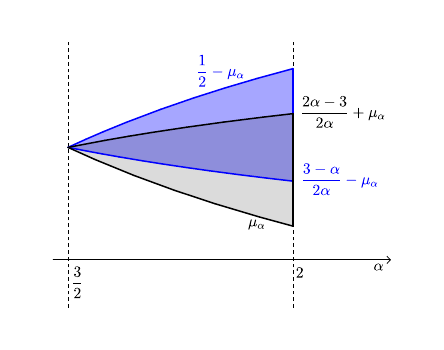}}
\hskip-30pt
          \subfloat[$\xi=0.97$]{\includegraphics[width=6.5cm]{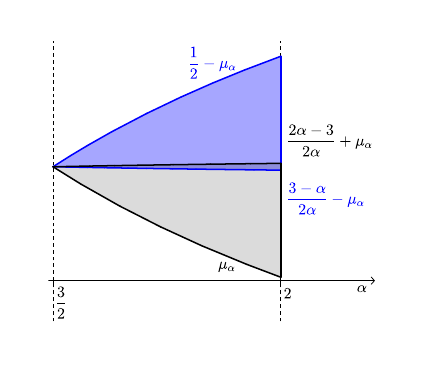}}
          \caption{Non-empty intersections}\label{figure}
        \end{figure}

\begin{remark}\label{re:bound}


In Theorem \ref{th:reg-l2}
we can take $\mu_\alpha=0$. In that case, we have
$u_1$, $\nabla u_2\in L^2(\Omega)$,
\eqref{eq:nabla-l2} holds for 
$\theta\in\big(0,\frac{2\alpha-3}{2\alpha}\big)$ while \eqref{eq:partial-l2} holds for 
$\theta\in\big(\frac{3-\alpha}{2\alpha},\frac12\big)$. 
We note that the intervals $\big(0,\frac{2\alpha-3}{2\alpha}\big)$ and $\big(\frac{3-\alpha}{2\alpha},\frac12\big)$ have no common points, because $\frac{2\alpha-3}{2\alpha}<\frac{3-\alpha}{2\alpha}$ thanks to $\alpha<2$.

To prove the trace regularity result, it is essential that the estimates
\eqref{eq:nabla-l2} and \eqref{eq:partial-l2} hold for the same value of $\theta$. 
For this reason we must choose a suitable value of $\mu_\alpha>0$ such that the intersection of the intervals
$\big(\mu_\alpha,\frac{2\alpha-3}{2\alpha}+\mu_\alpha\big)$ and $\big(\frac{3-\alpha}{2\alpha}-\mu_\alpha,\frac12-\mu_\alpha\big)$ is not empty. Consequently, we have to impose $\frac{3-\alpha}{2\alpha}-\mu_\alpha<\frac{2\alpha-3}{2\alpha}+\mu_\alpha$ and $\mu_\alpha<\frac12-\mu_\alpha$, hence we obtain the following condition $\frac{3(2-\alpha)}{4\alpha}<\mu_\alpha<\frac14$. 

To provide further clarification, we consider the expression $\mu_\alpha=\frac{3(2-\alpha)\xi}{4\alpha}+\frac{1-\xi}4$ with $\xi\in(0,1)$. As shown in Figure \ref{figure}, we can observe the non-empty intersection of the intervals for two distinct values of $\xi$.
\end{remark}

\section{Trace regularity results}\label{s:hidreg}

In this section we follow arguments similar to those implemented in \cite{K} for wave equations.
First, we single out some technical results that we will use later.

\begin{lemma}\label{le:tech0} 
For $u\in H^2(\Omega)$ and a vector field $h:\overline{\Omega}\to\mathbb{R}^N$ of class $C^1$ one has
\begin{multline}\label{eq:triangleuF00}
\int_\Omega\triangle u\, h\cdot\nabla u\, dx
=
\int_{\partial\Omega}\Big[\partial_\nu u\, h\cdot\nabla u
-\frac12h\cdot\nu |\nabla u|^2\Big]\, d\sigma
\\
-\sum_{i,j=1}^N\int_\Omega
\partial_i  h_j\partial_i u\partial_j u\, dx
+\frac12\int_{\Omega}
\sum_{j=1}^N \partial_jh_j\ |\nabla u|^2\, dx.
\end{multline}
\end{lemma}

\begin{proof} 
We integrate by parts to get
\begin{equation}\label{eq:triangleu00}
\int_\Omega\triangle u\, h\cdot\nabla u\, dx
=
\int_{\partial\Omega}
\partial_\nu u\, h\cdot\nabla u\, d\sigma
-\int_\Omega\nabla u
\cdot\nabla \big( h\cdot\nabla u\big)\, dx.
\end{equation}
Since
\begin{multline*}
\int_\Omega\nabla u\cdot\nabla \big( h\cdot\nabla u \big)\, dx
=\sum_{i,j=1}^N\int_\Omega
\partial_i u\, \partial_i ( h_j\partial_j u)\, dx
\\
=\sum_{i,j=1}^N\int_\Omega\partial_i u\, \partial_i  h_j\partial_j u\, dx
+\sum_{i,j=1}^N\int_\Omega h_j\, \partial_i u\partial_j ( \partial_i u)\, dx,
\end{multline*}
we evaluate the last term on the right-hand side again by an integration  by parts, so we obtain
\begin{align*}
\sum_{i,j=1}^N\int_\Omega
h_j\, \partial_i u \partial_j ( \partial_i u)\, dx
=&\frac12\sum_{j=1}^N\int_\Omega
h_j\, \partial_j \Big( \sum_{i=1}^N(\partial_i u)^2\Big)\, dx
\\
=&\frac12\int_{\partial\Omega}
h\cdot\nu |\nabla u|^2\, d\sigma
-\frac12\int_{\Omega}
\sum_{j=1}^N \partial_jh_j\ |\nabla u|^2\, dx.
\end{align*}
Finally, combining the previous two identities with \eqref{eq:triangleu00} yields \eqref{eq:triangleuF00}.
\end{proof}

Let us  keep in mind the symbol $\langle \cdot,\cdot\rangle_{-\theta,\theta}$, $\theta\in(0,1)$, denoting the duality
defined in \eqref{eq:duality}.

\begin{lemma}\label{le:tech} 
Assume $\alpha\in(1,2)$ and the weak solution $u$ of
\begin{equation}\label{eq:stato}
\sideset{}{_{0+}^{\alpha}}{\mathop D}u(t,x) =\triangle u (t,x)
\quad
\text{in}
\ \
(0,T)\times\Omega
\end{equation}
belonging to $L^2(0,T;H^2(\Omega)\cap H^1_0(\Omega))$ with $\sideset{}{_{0+}^{\alpha}}{\mathop D}u\in L^2(0,T;L^2(\Omega))$.
Then, for a vector field $h:\overline{\Omega}\to\mathbb{R}^N$ of class $C^1$ and $\theta\in(0,1)$ the following identity holds true$:$
\begin{multline}\label{eq:identity}
\int_{\partial\Omega}
\Big[\partial_\nu u(t,\sigma)\,  h\cdot \nabla u(t,\sigma)
-\frac12 h\cdot\nu \big|\nabla u(t,\sigma)\big|^2\Big]\, d\sigma
\\
=\langle \sideset{}{_{0+}^{\alpha}}{\mathop D}u(t,\cdot),h\cdot \nabla u(t,\cdot)\rangle_{-\theta,\theta}
+\sum_{i,j=1}^N\int_\Omega\partial_i  h_j (x)\partial_i u(t,x)\partial_j u(t,x)\, dx
\\
-\frac12\int_{\Omega}\sum_{j=1}^N \partial_jh_j\ | \nabla u(t,x)|^2\, dx, \quad\text{a.e.} \ t\in (0,T).
\end{multline}

\end{lemma}

\begin{proof} 
Fix $\theta\in(0,1)$, 
by means of the duality
$\langle \cdot,\cdot\rangle_{-\theta,\theta}$, see \eqref{eq:duality}, for a.e. $t\in (0,T)$
we evaluate each term of the equation \eqref{eq:stato} in
$
h\cdot \nabla u(t,\cdot)   
$
to get
\begin{equation*}
\langle D_{0+}^{\alpha}u(t,\cdot),h\cdot\nabla u(t,\cdot)\rangle_{-\theta,\theta}
=\langle \triangle u(t,\cdot),h\cdot\nabla u(t,\cdot)\rangle_{-\theta,\theta}.
\end{equation*}
Thanks to the regularity of the weak solution $u$ and \eqref{eq:-theta} the term on the right-hand side of the previuos equation can be written as a scalar product in $L^2(\Omega)$, so we have
\begin{equation}\label{eq:identity1}
\langle D_{0+}^{\alpha}u(t,\cdot),h\cdot\nabla u(t,\cdot)\rangle_{-\theta,\theta}
=\int_\Omega \triangle u(t,x)\, h\cdot\nabla u(t,x)\, dx.
\end{equation}
By evaluating the term
$
\int_\Omega \triangle u(t,x)\, h\cdot\nabla u(t,x)\, dx
$
with Lemma \ref{le:tech0}, 
from \eqref{eq:triangleuF00} we deduce
\begin{multline*}
\int_{\partial\Omega}
\Big[\partial_\nu u(t,\sigma)\,  h\cdot \nabla u(t,\sigma)
-\frac12 h\cdot\nu \big|  \nabla u(t,\sigma)\big|^2\Big]\, d\sigma
\\
=\int_\Omega \triangle u(t,x)
h\cdot\nabla u(t,x)\ dx
+\sum_{i,j=1}^N\int_\Omega
\partial_i  h_j (x)\partial_i u(t,x)\partial_j u(t,x)\, dx
\\
-\frac12\int_{\Omega}
\sum_{j=1}^N \partial_jh_j\ | \nabla u(t,x)|^2\, dx.
\end{multline*}
We conclude from \eqref{eq:identity1} that \eqref{eq:identity} is proved.
\end{proof}  
\begin{theorem}\label{th:hidstr} 
Let  $u_1\in H^1_0(\Omega)$, $u_2\in H^2(\Omega)\cap H^1_0(\Omega)$ and $u$ the weak solution of
\begin{equation}\label{eq:cauchy1}
\begin{aligned}
\sideset{}{_{0+}^{\alpha}}{\mathop D}u(t,x)
&=\Delta u(t,x)
\qquad
(t,x)\in (0,T)\times\Omega,
\\
u(t,x)&=0 
\qquad\qquad (t,x)\in (0,T)\times\partial\Omega,
\end{aligned}
\end{equation}
satisfying the  initial data 
\begin{equation} \label{eq:indata1}
\sideset{}{_{0+}^{\alpha-1}}{\mathop D}u(0,\cdot)=u_{1},\qquad \sideset{}{_{0+}^{\alpha-2}}{\mathop D}u(0,\cdot)=u_{2}.
\end{equation}
Then, for any $T>0$ there exists a constant $C=C(T)$ such that for any $\mu_\alpha\in\big(\frac{3(2-\alpha)}{4\alpha},\frac14\big)$ $u$ satisfies the inequality
\begin{equation}\label{eq:hidden-alpha}
\int_0^T\int_{\partial\Omega} \big|\partial_\nu u\big|^2\, d\sigma dt
\le C\big(\lVert u_1\rVert_{D((-\Delta)^{\mu_\alpha})}
+\lVert\nabla u_2\rVert_{D((-\Delta)^{\mu_\alpha})}\big).
\end{equation}
\end{theorem}

\begin{proof} 
First, we note that by Theorem \ref{th:exist}--(ii) the weak solution $u$ to \eqref{eq:weakp} given by \eqref{eq:def-u0} belongs to $L^2(0,T;H^2(\Omega)\cap H^1_0(\Omega))$ and $\sideset{}{_{0+}^{\alpha}}{\mathop D}u(t,\cdot)\in L^2(0,T;L^2(\Omega))$. In particular, the normal derivative $\partial_\nu u$ is well defined.

We employ the  identity in Lemma \ref{le:tech} with a suitable choice of the vector field $h$. 
Indeed, we take a vector field $h\in C^1(\overline{\Omega};\mathbb{R}^N)$ satisfying the condition
\begin{equation*}
h=\nu
\qquad
\text{on}\quad\partial\Omega
\end{equation*}
(see e.g. \cite {K} for the existence of such vector field $h$). First we consider the identity \eqref{eq:identity}.
Since
\begin{equation*}
\nabla u=(\partial_\nu u)\nu
\quad
\text{on}
\quad (0,T)\times\partial\Omega\,,
\end{equation*}
(see e.g. \cite[Lemma 2.1]{MM} for a detailed proof)
the left-hand side of \eqref{eq:identity} becomes
\begin{equation*}
\frac12\int_{\partial\Omega} \big|\partial_\nu u(t,\sigma)\big|^2\, d\sigma.
\end{equation*}
If we
integrate  \eqref{eq:identity} over $[0,T]$, then we obtain
\begin{multline*}
\int_0^T\int_{\partial\Omega}
\big|\partial_\nu u(t,\sigma)\big|^2\, d\sigma dt
=2\int_0^T\langle \sideset{}{_{0+}^{\alpha}}{\mathop D}u(t,\cdot),h\cdot \nabla u(t,\cdot)\rangle_{-\theta,\theta}\ dt
\\
+2\sum_{i,j=1}^N\int_0^T\int_\Omega
\partial_i  h_j (x)\partial_i u(t,x)\partial_j u(t,x)\, dx dt
-\int_0^T\int_{\Omega}
\sum_{j=1}^N \partial_jh_j\ | \nabla u(t,x)|^2\, dxdt.
\end{multline*}
Since $h\in C^1(\overline{\Omega};\mathbb{R}^N)$ from the above inequality we get
\begin{equation}\label{eq:norml2beta}
\big\|\partial_\nu u\big\|_{L^2(0,T;L^2(\partial\Omega))}
\le C\Big(\big\|\sideset{}{_{0+}^{\alpha}}{\mathop D}u\big\|_{L^2(0,T;D(A^{-\theta}))}
+\big\|\nabla u\big\|_{L^2(0,T;D(A^{\theta}))}\Big),
\end{equation}
for some constant $C>0$.
To apply Theorem \ref{th:reg-l2}, since  $\mu_\alpha\in\big(\frac{3(2-\alpha)}{4\alpha},\frac14\big)$ 
we can choose 
\begin{equation*}
\theta\in \Big(\mu_\alpha,\frac{2\alpha-3}{2\alpha}+\mu_\alpha\Big)\cap\Big(\frac{3-\alpha}{2\alpha}-\mu_\alpha,\frac12-\mu_\alpha\Big)
\end{equation*}
see Remark \ref{re:bound}. So, we  get
$\sideset{}{_{0+}^{\alpha}}{\mathop D}u\in L^2(0,T;D(A^{-\theta}))$ and
$\nabla u\in L^2(0,T;D(A^{\theta}))$,
hence from \eqref{eq:norml2beta}, \eqref{eq:partial-l2} and \eqref{eq:nabla-l2}  we deduce
\eqref{eq:hidden-alpha}. 
\end{proof} 

\begin{theorem}\label{th:hidalpha}  
Let  $u_1, \nabla u_2\in D((-\Delta)^{\mu_\alpha})$ with $\frac{3(2-\alpha)}{4\alpha}<\mu_\alpha<\frac14$. If $u$ is the weak solution of
\eqref{eq:cauchy1}-\eqref{eq:indata1},
then we define the normal derivative $\partial_\nu u$ of $u$ such that  for any $T>0$ we have
\begin{equation}\label{eq:hidden-alpha0}
\int_0^T\int_{\partial\Omega} \big|\partial_\nu u\big|^2\, d\sigma dt
\le C\big(\lVert u_1\rVert_{D((-\Delta)^{\mu_\alpha})}
+\lVert\nabla u_2\rVert_{D((-\Delta)^{\mu_\alpha})}\big),
\end{equation}
for some constant $C=C(T)$ independent of the initial data.
\end{theorem}
\begin{proof} 
For $u_1\in H^1_0(\Omega)$ and $u_2\in H^2(\Omega)\cap H^1_0(\Omega)$, denoted by $u$ the weak solution of problem
\eqref{eq:cauchy1}-\eqref{eq:indata1},  thanks to Theorem \ref{th:hidstr}, the inequality
\eqref{eq:hidden-alpha0} holds for any $T>0$.
By density there exists a unique continuous linear map
\begin{equation*}
{\mathcal L}:D((-\Delta)^{\mu_\alpha})\times D((-\Delta)^{\frac12+\mu_\alpha})\to L^2_{loc}((0,\infty);L^2(\partial\Omega))
\end{equation*}
such that
\begin{equation*}
{\mathcal L}(u_1,u_2)=\partial_\nu u
\qquad\forall(u_1,u_2)\in H^1_0(\Omega)\times\big(H^2(\Omega)\cap H^1_0(\Omega)\big) 
\end{equation*}
and
\begin{equation*}
\int_0^T\int_{\partial\Omega} \big|{\mathcal L}(u_1,u_2)\big|^2d\sigma dt
\le  C\big(\lVert u_1\rVert_{D((-\Delta)^{\mu_\alpha})}
+\lVert\nabla u_2\rVert_{D((-\Delta)^{\mu_\alpha})}\big),
\end{equation*}
for any $u_1, \nabla u_2\in D((-\Delta)^{\mu_\alpha})$.
Consequently, we adopt the standard notation $\partial_\nu u$ in place of  ${\mathcal L}(u_1,u_2)$, which leads to \eqref{eq:hidden-alpha0} holding. 
\end{proof} 
\begin{remark}\label{re:bound1}
We observe that the assumption 
$u_1, \nabla u_2\in D((-\Delta)^{\mu_\alpha})$ with $\frac{3(2-\alpha)}{4\alpha}<\mu_\alpha<\frac14$ is weaker than assuming 
$u_1\in H^1_0(\Omega)$ and $u_2\in H^2(\Omega)\cap H^1_0(\Omega)$, due to
\begin{equation*}
H^1_0(\Omega)=D((-\Delta)^{\frac12})\subset D((-\Delta)^{\frac14})\subset D((-\Delta)^{\mu_\alpha}).
\end{equation*}
\end{remark} 

\section{Further properties}

In this section we establish the decay rate of so-called energy for the Riemann-Liou\-ville problem. 

Next, we investigate the duality existing between 
the Caputo problem and the Riemann-Liouville problem following an approach similar to that of J.-L. Lions \cite{Lio3} for the case of the wave equation. 


\subsection{Decay rate of energy}

The analysis of the decay rate in $L^2$-norm started from the case $\alpha\in(0,1)$ for the problem involving the Caputo fractional derivative instead of the Riemann-Liouville fractional derivative, see \cite{VZ} and also \cite{KRY}. 
In \cite{SY} the authors established the decay rate in $L^2$-norm when $\alpha\in(1,2)$.
This result was improved in \cite{Ya}, in the case $\alpha\in(0,1)\cup(1,2)$. 

Acting similarly to \cite{Ya}, we prove the decay rate for the weak solution of the problem
\begin{equation}\label{eq:decay}
\begin{aligned}
\sideset{}{_{0+}^{\alpha}}{\mathop D}u(t,x)
&=\Delta u(t,x)
\qquad
(t,x)\in (0,T)\times\Omega,
\\
u(t,x)&=0 
\qquad\qquad (t,x)\in (0,T)\times\partial\Omega,
\\
\sideset{}{_{0+}^{\alpha-1}}{\mathop D}u(0,\cdot)&=u_{1},\qquad \sideset{}{_{0+}^{\alpha-2}}{\mathop D}u(0,\cdot)=u_{2}.
\end{aligned}
\end{equation}
\begin{theorem}\label{th:}
 The weak solution 
 \begin{multline}\label{eq:}
u(t,x)
=\sum_{n=1}^\infty\big[ \langle u_1,e_n\rangle t^{\alpha-1}E_{\alpha,\alpha}(-\lambda_nt^\alpha)
+\langle u_2,e_n\rangle t^{\alpha-2} E_{\alpha,\alpha-1}(-\lambda_nt^\alpha)\big]e_n(x)
\end{multline}
of \eqref{eq:decay} satisfies
\begin{equation}\label{eq:decay1}
\|u(t,\cdot)\|_{H^2(\Omega)}\le \frac{C}{t}\|u_1\|+\frac{C}{t^{2}}\|u_2\|,
\qquad\text{for any}\ t>0.
\end{equation}
\end{theorem}
\begin{proof} 
Since
\begin{equation*}
\lVert u(t,\cdot)\rVert_{H^2(\Omega)}^2
=\sum_{n=1}^\infty \lambda_n^2 \big| \langle u_1,e_n\rangle t^{\alpha-1}E_{\alpha,\alpha}(-\lambda_nt^\alpha)
+\langle u_2,e_n\rangle t^{\alpha-2} E_{\alpha,\alpha-1}(-\lambda_nt^\alpha)\big|^2,
\end{equation*}
thanks to \eqref{eq:stimeE} we have for $t>0$
\begin{equation*}
\lVert u(t,\cdot)\rVert_{H^2(\Omega)}^2
\le
Ct^{-2}\sum_{n=1}^\infty  \big| \langle u_1,e_n\rangle\big|^2 \frac{\lambda_n^2t^{2\alpha}}{1+\lambda_n^2t^{2\alpha}}
+Ct^{-4}\sum_{n=1}^\infty  \big|\langle u_2,e_n\rangle\big|^2 \frac{\lambda_n^2t^{2\alpha}} {1+\lambda_n^2t^{2\alpha}}.
\end{equation*}
Therefore
\begin{equation*}
\|u(t,\cdot)\|_{H^2(\Omega)}^2\le \frac{C}{t^2}\|u_1\|^2+\frac{C}{t^{4}}\|u_2\|^2,
\qquad\text{for any}\ t>0,
\end{equation*}
that is \eqref{eq:decay1}.
\end{proof} 
We observe that the estimate \eqref{eq:decay1} does not depend on $\alpha$.

\subsection{Duality}

Let $T>0$ and $\Omega$ be a bounded domain of class $C^2$ in $\mathbb{R}^N$, $N\ge1$, with boundary $\partial\Omega$.
We consider the Caputo fractional diffusion-wave equation with $\alpha\in(1,2)$:
\begin{equation}\label{eq:problem-u}
\sideset{^{C}}{_{0+}^{\alpha}}{\mathop D}u(t,x)
=\Delta u(t,x)
\qquad
(t,x)\in (0,T)\times\Omega,
\end{equation}
with null initial conditions 
\begin{equation}
u(0,x)=u_{t}(0,x)=0\qquad  x\in\Omega,
\end{equation} 
and boundary conditions, 
\begin{equation}\label{eq:bound-u1}
u(t,x)=g(t,x) \quad (t,x)\in (0,T)\times\partial\Omega.
\end{equation}
We are interested in determining the dual system of the fractional differential problem \eqref{eq:problem-u}-\eqref{eq:bound-u1} also due to the uniqueness properties of the dual system. 

For the meaning and study of duality in the context of control theory, we refer to the seminal works \cite{Lio3,R} and references therein. 

In relation to the fractional differential case, one of the main difficulties is determining the function spaces in which to set the problem and an appropriate meaning for the solution. 

We introduce the {\it adjoint} system of \eqref{eq:problem-u}-\eqref{eq:bound-u1} as
\begin{equation}\label{eq:adjoint}
\begin{cases}
\displaystyle 
\sideset{}{_{T-}^{\alpha}}{\mathop D}w(t,x)
=\Delta w(t,x)
\qquad
(t,x)\in (0,T)\times\Omega,
\\
\\
w(t,x)=0 
\qquad (t,x)\in (0,T)\times\partial\Omega,
\end{cases}
\end{equation}
with  final data 
\begin{equation}\label{eq:final}
\sideset{}{_{T-}^{\alpha-1}}{\mathop D}w(T,\cdot)=w_{1},\qquad \sideset{}{_{T-}^{\alpha-2}}{\mathop D}w(T,\cdot)=w_{2}.
\end{equation} 
Set $v(t,x)=w(T-t,x)$, for any $\beta>0$ we have 
\begin{multline*}
\int_t^T (s-t)^{\beta-1}w(s,x)\, ds
\\
=\int_t^T (s-t)^{\beta-1}v(T-s,x)\, ds
=\int_0^{T-t} (T-t-\sigma)^{\beta-1}v(\sigma,x)\, d\sigma.
\end{multline*}
Therefore
\begin{equation*}
I_{T-}^{\beta}w(t,x)=I_{0+}^{\beta}v(T-t,x)
\qquad\beta>0,
\end{equation*}
and hence
\begin{equation*}
\begin{aligned}
\sideset{}{_{T-}^{\alpha}}{\mathop D}w(t,x)&=\sideset{}{_{0+}^{\alpha}}{\mathop D}v(T-t,x),
\\
\sideset{}{_{T-}^{\alpha-1}}{\mathop D}w(T,\cdot)&=\sideset{}{_{0+}^{\alpha-1}}{\mathop D}v(0,\cdot),
\quad
\sideset{}{_{T-}^{\alpha-2}}{\mathop D}w(T,\cdot)=\sideset{}{_{0+}^{\alpha-2}}{\mathop D}v(0,\cdot).
\end{aligned}
\end{equation*}
The above considerations ensure that the backward problem \eqref{eq:adjoint}-\eqref{eq:final} is equivalent to a forward problem of the type \eqref{eq:cauchy1}-\eqref{eq:indata1}. 
Thanks to Theorem \ref{th:hidalpha}, by appropriately choosing the regularity of the final data $w_1$ and $w_2$, the solution  $w$ of \eqref{eq:adjoint}--\eqref{eq:final} will be sufficiently smooth to guarantee that $\partial_\nu w$ is well-defined, allowing us to take $g=\partial_\nu w$ in \eqref{eq:bound-u1}.
The  nonhomogeneous problem \eqref{eq:problem-u}-\eqref{eq:bound-u1} can be written in the form
\begin{equation}\label{eq:phi}
\left \{\begin{array}{l}\displaystyle
\sideset{^{C}}{_{0+}^{\alpha}}{\mathop D}u(t,x)
=\Delta u(t,x)
\qquad (t,x)\in (0,T)\times\Omega,
\\
\\
u(0,x)=u_{t}(0,x)=0\qquad  x\in\Omega,
\\
\displaystyle
u(t,x)=\partial_\nu w(t,x), \quad (t,x)\in (0,T)\times\partial\Omega.
\end{array}\right .
\end{equation} 

In a similar way to \cite{Lio3} we show duality by introducing a suitable operator.
\begin{theorem}
Let $\Psi$ be a linear operator defined as 
\begin{equation}\label{eq:psi0}
\Psi(w_{1},w_{2})=\big(u(T,\cdot),-u_t(T,\cdot)\big),
\end{equation}
for $w_1, \nabla w_2\in D((-\Delta)^{\mu_\alpha})$ with $\frac{3(2-\alpha)}{4\alpha}<\mu_\alpha<\frac14$, where $u$ and $w$ are the solution of \eqref{eq:phi} and \eqref{eq:adjoint}-\eqref{eq:final} respectively.

Then
\begin{equation}\label{eq:psi}
\langle\Psi(w_{1},w_{2}),(w_{1},w_{2})\rangle
=\int_0^T\int_{\partial\Omega}\big|\partial_{\nu}w(t,x)\big|^2 \, dx dt. 
\end{equation}
\end{theorem}
\begin{proof}
We multiply the  equation in (\ref{eq:phi}) by the solution $w(t,x)$ of the adjoint system \eqref{eq:adjoint}-\eqref{eq:final} and integrate 
on $(0,T)\times\Omega$, that is
\begin{equation}\label{eq:mult}
\int_0^T\int_{\Omega}\sideset{^{C}}{_{0+}^{\alpha}}{\mathop D}u(t,x)w(t,x)\, dx dt
=\int_0^T\int_{\Omega}\Delta u(t,x)w(t,x)\, dx dt.
\end{equation}
Applying \eqref{eq:change+-} we obtain for any $x\in\Omega$
\begin{multline}\label{eq:intp}
\int_0^T \sideset{^{C}}{_{0+}^{\alpha}}{\mathop D}u(t,x) w(t,x)\, dt
\\
=\int_0^TI_{0+}^{2-\alpha}u_{tt}(t,x) w(t,x)\, dt
=\int_0^Tu_{tt}(t,x) I_{T-}^{2-\alpha}w(t,x)\, dt.
\end{multline}
Integrating twice by parts and taking into account that $u(0,\cdot)=u_t(0,\cdot)=0$, we have
\begin{multline}\label{eq:intp1}
\int_0^Tu_{tt}(t,x)I_{T-}^{2-\alpha}w(t,x)\, dt
\\
=u_t(T,x)I_{T-}^{2-\alpha} w(T,x)-\int_0^Tu_t(t,x)\frac{\partial}{\partial t}I_{T-}^{2-\alpha}w(t,x)\, dt
\\
=u_t(T,x)I_{T-}^{2-\alpha} w(T,x)-u(T,x)\frac{\partial}{\partial t}I_{T-}^{2-\alpha}w(T,x)
+\int_0^Tu(t,x)\sideset{}{_{T-}^{\alpha}}{\mathop D}w(t,x)\, dt.
\end{multline}
Since $-1<\alpha-2<0$, we have
$
I_{T-}^{2-\alpha}w(T,x)
=\sideset{}{_{T-}^{\alpha-2}}{\mathop D}w(T,x)
$,
see \eqref{eq:RL-T0}.
Moreover, $0<\alpha-1<1$, so by \eqref{eq:alpha-1} we have
\begin{equation*}
\frac{\partial}{\partial t}I_{T-}^{2-\alpha}w(T,x)
=\sideset{}{_{T-}^{\alpha-1}}{\mathop D}w(T,x).
\end{equation*}
Combining \eqref{eq:intp} with \eqref{eq:intp1} yields
\begin{multline}\label{eq:change1}
\int_0^T\sideset{^{C}}{_{0+}^{\alpha}}{\mathop D}u(t,x)w(t,x)\, dt
\\
=u_t(T,x)\sideset{}{_{T-}^{\alpha-2}}{\mathop D}w(T,x)-u(T,x)\sideset{}{_{T-}^{\alpha-1}}{\mathop D}w(T,x)
+\int_0^Tu(t,x)\sideset{}{_{T-}^{\alpha}}{\mathop D}w(t,x)\, dt.
\end{multline}
On the other hand, integrating twice by parts with respect to the variable $x$ we have
\begin{equation}\label{eq:change2}
\int_{\Omega}\Delta u(t,x)w(t,x)\, dx=
\int_\Omega u(t,x)\triangle w(t,x)\, dx
-\int_{\partial\Omega}u(t,x)\partial_{\nu}w(t,x) \, dx.
\end{equation}
Putting \eqref{eq:change1}  and  \eqref{eq:change2} into \eqref{eq:mult} we get
\begin{multline*}
\int_0^T\int_\Omega u(t,x)\sideset{}{_{T-}^{\alpha}}{\mathop D}w(t,x) \, dx dt 
\\
+\int_\Omega u_t(T,x)\sideset{}{_{T-}^{\alpha-2}}{\mathop D} w(T,x)\, dx
-\int_\Omega u(T,x)\sideset{}{_{T-}^{\alpha-1}}{\mathop D}w(T,x)\, dx
\\
=\int_0^T\int_\Omega u(t,x)\triangle w(t,x) \, dx dt
-\int_0^T\int_{\partial\Omega}u(t,x)\partial_{\nu}w(t,x) \, dx dt.
\end{multline*}
We recall that $w$ is the solution of the adjont problem \eqref{eq:adjoint}:
keeping  
$
\sideset{}{_{T-}^{\alpha}}{\mathop D}w
=\Delta w
$
in mind, 
from the above equation we have
\begin{multline*}
\int_\Omega u_t(T,x)\sideset{}{_{T-}^{\alpha-2}}{\mathop D} w(T,x)\, dx
-\int_\Omega u(T,x)\sideset{}{_{T-}^{\alpha-1}}{\mathop D}w(T,x)\, dx
\\
=-\int_0^T\int_{\partial\Omega}u(t,x)\partial_{\nu}w(t,x) \, dx dt.
\end{multline*}
By the final data 
\eqref{eq:final} of $w$ and the boundary conditions satisfied by $u$, see \eqref{eq:phi} we deduce that
\begin{multline*}
\langle\Psi(w_{1},w_{2}),(w_{1},w_{2})\rangle
\\
=
\int_\Omega  u(T,x) w_1(x) \, dx
-\int_\Omega  u_t(T,x)w_2(x) \, dx
=\int_0^T\int_{\partial\Omega}\big|\partial_{\nu}w(t,x)\big|^2 \, dx dt,
\end{multline*}
hence \eqref{eq:psi} is proved.
\end{proof}


\begin{funding}
This research was partially supported by Progetto Ateneo 2021 Sapienza Universit\`a di Roma ``Effect of fractional time derivatives in evolution equations''.
\end{funding}


\end{document}